\documentclass[11pt]{article}
\usepackage[utf8]{inputenc}
\usepackage{amsfonts}
\usepackage{amssymb}
\usepackage{fullpage}
\usepackage{lipsum}
\usepackage{comment}
\usepackage[]{algorithm2e}
\usepackage{amssymb}
\usepackage{tipa}
\usepackage[skip=20pt]{parskip}
\usepackage[shortlabels]{enumitem}
\usepackage{fancyhdr}
\usepackage{amsthm}
\usepackage{amsmath}
\usepackage{mathrsfs} 
\usepackage{url}
\usepackage{biblatex}
\renewbibmacro{in:}{}
\usepackage{verbatim}
\usepackage{setspace}
\usepackage{mathtools}
\usepackage{tikz}
\usepackage{multicol}
\usepackage{xfrac}
\usepackage{bm}
\usepackage{yfonts}
\usepackage{underscore}
\usepackage[super]{nth}
\usepackage{textcomp}
\usepackage{pdfpages}
\usepackage[UKenglish]{babel}
\usepackage[UKenglish]{isodate}
\usepackage{dsfont}
\usepackage{graphicx}
\usepackage[rightcaption]{sidecap}
\usepackage{wrapfig}
\usepackage[bookmarks=true]{hyperref}

\newcommand{\mathleft}{\@fleqntrue\@mathmargin0pt}
\newcommand{\mathcenter}{\@fleqnfalse}

\theoremstyle{plain}
\newtheorem{theorem}{Theorem}[section]
\newtheorem*{conjecture}{Conjecture}

\newtheorem*{example*}{Example}

\theoremstyle{definition}
\newtheorem{definition}[theorem]{Definition}
\newtheorem{corollary}[theorem]{Corollary}
\newtheorem{lemma}[theorem]{Lemma}
\newtheorem{proposition}[theorem]{Proposition}
\newtheorem{example}[theorem]{Example}
\newtheorem{remark}[theorem]{Remark}
\newtheorem{code}{Code}
\newtheorem{question}[theorem]{Question}
\newtheorem*{acknowledgement}{Acknowledgement}

\makeatletter  
\def\colvec#1{\expandafter\colvec@i#1,,,,,,,,,\@nil}
\def\colvec@i#1,#2,#3,#4,#5,#6,#7,#8,#9\@nil{%
  \ifx$#2$ \begin{pmatrix}#1\end{pmatrix} \else
    \ifx$#3$ \begin{pmatrix}#1\\#2\end{pmatrix} \else
      \ifx$#4$ \begin{pmatrix}#1\\#2\\#3\end{pmatrix}\else
        \ifx$#5$ \begin{pmatrix}#1\\#2\\#3\\#4\end{pmatrix}\else
          \ifx$#6$ \begin{pmatrix}#1\\#2\\#3\\#4\\#5\end{pmatrix}\else
            \ifx$#7$ \begin{pmatrix}#1\\#2\\#3\\#4\\#5\\#6\end{pmatrix}\else
              \ifx$#8$ \begin{pmatrix}#1\\#2\\#3\\#4\\#5\\#6\\#7\end{pmatrix}\else
                 \PackageError{Column Vector}{The vector you tried to write is too big, use pmatrix instead}{Try using the pmatrix environment}
              \fi
            \fi
          \fi
        \fi
      \fi
    \fi
  \fi 
}  
\makeatother

\emergencystretch=\maxdimen
\graphicspath{ {./pictures/} }

\title{The Knutson Index of the Representation Ring} 
\author{Diego Martín Duro}
\date{\vspace{-1ex}}

\begin{document}

\maketitle

\begin{abstract}
In this paper, we study if, for a given simple module over a Hopf algebra, there exists a virtual module such that their tensor product is the regular module. This is related to a conjecture by Donald Knutson, later disproved and refined by Savitskii, stating that for every irreducible character of a finite group, there exists a virtual character such that their tensor product is the regular character. We also introduce the Knutson Index as a measure of Knutson's Conjecture failure, discuss its algebraic properties and present counter-examples to Savitskii's Conjecture.
\end{abstract}

\section{Introduction}
The set of virtual complex characters over a group, i.e., integer linear combinations of irreducible complex characters, forms a ring with direct sum as addition and tensor product as multiplication. Donald Knutson conjectured in 1973 that every irreducible character $\chi$ of a finite group $G$ is regular invertible, this means that there exists a virtual character $\varphi \in \mathbb{Z}[\text{Irr}(G)]$ such that $\chi \otimes \varphi = \rho_\text{reg}$, where $\rho_\text{reg}$ is the regular character \cite{knutson}. In 1992 Savitskii observed that this conjecture failed for $SL_2(\mathbb{F}_5)$ and refined Knutson's Conjecture by introducing a pre-order on the set of virtual complex characters and claimed that an irreducible character $\chi$ is regular invertible if and only if it is minimal with respect to this pre-order \cite{savitskii}.

The main contributions of the paper are introducing the Knutson Index of an irreducible complex character $\chi$ of a finite group $G$ as the smallest positive integer $n$ such that there exists a virtual character $\varphi \in \mathbb{Z}[\text{Irr}(G)]$ such that $\chi \otimes \varphi = n \rho_\text{reg}$ and showing that this index can be computed by a non-standard restriction of $\chi$ to the Sylow subgroups of $G$. The other key result is the disproval of both directions of Savitskii's Conjecture.

We start by introducing the more general (left) Knutson Index of a module over a bialgebra, which measures the failure of Knutson's Conjecture and we then define bialgebras of Knutson type as the bialgebras that satisfy the generalisation of Knutson's Conjecture. Next, we present some properties of this index for modules over Hopf algebras. 

In the second section, we focus on group algebras. We find, using Sage, more counter-examples to Knutson's Conjecture and compute their Knutson Indices. We also show that all group algebras of an extraspecial or dihedral group are of Knutson type. Next, we introduce the Savitskii Restriction to a Sylow $p$-subgroup and the Savitskii Induced Inverse to show that the Knutson Index of a module is the product of the Knutson Indices of its Savitskii Restrictions to all Sylow $p$-subgroup. This is a very useful tool that allows us to compute the Knutson Indices for more families of group algebras and check whether they are of Knutson type. For example, we prove Savitskii's Criterion on the validity of Knutson's Conjecture for $\mathbb{C}[{\rm SL}_2(\mathbb{F}_q)]$ and $\mathbb{C}[{\rm PSL}_2(\mathbb{F}_q)]$ and that a group algebra whose dimension is a product of distinct primes is always of Knutson type.

We then turn the focus of the Knutson Index to the more general case of semisimple Hopf algebras. We start by recalling some properties of semisimple Hopf algebras and considering some examples. This helps us to find some conditions under which a semisimple Hopf algebra is always of Knutson type. For example, we show that if all simple modules of a Hopf algebra $H$ are one-dimensional or $n$-dimensional, then $H$ is of Knutson type. This leads to the corollaries that all semisimple Hopf algebras of dimension $p^3$ and $p^4$ are of Knutson type and that all semisimple Hopf algebras of dimension up to $31$, except in one case of dimension $24$, are also of Knutson type.

In the last section, we disprove the refinement of the conjecture that Savitskii gave after disproving Knutson's Conjecture. We start by showing that the dicyclic group $\text{Dic}_6$ of order $24$ has an irreducible character that is regular invertible but not minimal. We then show that the Savitskii Restriction does not preserve minimality of irreducible characters. This justifies our observation that $\text{D}_{12} \rtimes S_3$ has an irreducible character that is minimal but not regular invertible, despite all minimal irreducible characters of its Sylow $p$-subgroups being regular invertible.

\begin{acknowledgement}
I would like to express my sincere appreciation towards my PhD supervisor Professor Dmitriy Rumynin for introducing me to this interesting topic and for his continued guidance and corrections that have contributed to this paper. This work was supported by the UK Engineering and Physical Sciences Research Council (EPSRC) grant EP/T51794X/1.
\end{acknowledgement}

\section{Knutson Index}
\subsection{Definition and Properties}

\begin{definition}[Knutson Index]
Let $A$ be a bialgebra over a field $k$ and $M$ an $A$-module. The (left) Knutson Index of $M$ is the smallest strictly positive integer $n$ such that
$$ M \otimes N \cong n A $$
for some virtual $A$-module $N$. We denote this by $\mathcal{K}_L(M)$ and refer to $N$ as the (left) $n$-regular inverse of $M$. If no such $n$ exists, we say that Knutson Index of $M$ is infinite.
\end{definition}

The analogous definition of the right Knutson Index does not necessarily agree with this one as the representation ring of a bialgebra is not necessarily commutative. However, the left and right Knutson Index of all cases considered in this paper agree. We will therefore refer to the left Knutson Index of a module just as Knutson Index. We propose the following two questions.

\begin{question}
   Do the left and right Knutson Indices agree for any bialgebra?
\end{question}

\begin{question}
   Do the left and right Knutson Indices agree for any Hopf algebra?  
\end{question}

Let us now introduce the following terminology for bialgebras that satisfy the generalisation of Knutson's Conjecture.

\begin{definition}
If all simple modules over a bialgebra $A$ over a field $k$ have trivial Knutson Index, then we say that $A$ is a bialgebra of Knutson type.
\end{definition}

The notion of the Knutson Index is also related to Kaplansky's Sixth Conjecture. This conjecture asserts that if $A$ is a semisimple Hopf algebra over an algebraically closed field $k$ and $M$ is a simple $A$-module, then the dimension of $M$ divides the dimension of $A$ \cite{kaplansky1975bialgebras}. The next result follows.

\begin{proposition}
All Hopf algebras of Knutson type satisfy Kaplansky's Sixth Conjecture.
\end{proposition}

Let us now discuss some preliminary properties of the Knutson Index. The following proposition shows that the Knutson Index of a module over a Hopf algebra with invertible antipode, and in particular over any finite dimensional Hopf algebra, is well defined.

\begin{proposition} \label{prop divides}
If $H$ is a Hopf algebra with invertible antipode and $M$ is a finite dimensional module over $H$, then $\mathcal{K}(M)$ is finite and divides the dimension of $M$.
\end{proposition}

\begin{proof}
As the antipode of $H$ is invertible, we have that $M \otimes H \cong \text{dim}(M) \ H$. So it follows that $\mathcal{K}(M) \leq \text{dim}(M)$. Let $N$ be a $\mathcal{K}(M)$-regular inverse and $d$ the greatest common divisor of $\mathcal{K}(M)$ and $\text{dim}(M)$. Then there exist integers $a$ and $b$ such that 
$$ a \ \mathcal{K}(M) + b \ \text{dim}(M) = d $$
Hence
$$ M \otimes (a N + b H) \cong a \cdot \mathcal{K}(M) \ H + b \cdot \text{dim}(M) \ H = d \ H $$
We have constructed a $d$-regular inverse of $M$ and it follows by definition of the Knutson Index that $\mathcal{K}(M) \leq d$ and therefore $\mathcal{K}(M) = d$. We conclude that $\mathcal{K}(M)$ divides the dimension of $M$.
\end{proof}

We can define the Knutson Index for comodules analogously with the corresponding tensor product of comodules. In the following example, we show that all simple comodules over $\mathbb{C}[SL_2]$, the Hopf algebra of algebraic functions on $SL_2$, have trivial Knutson Index. Note that the duality between $\mathbb{C}[SL_2]$ and $\mathcal{U}(\mathfrak{sl}_2)$, the universal enveloping algebra of the special linear Lie algebra $\mathfrak{sl}_2$, induces a correspondence between the simple $\mathbb{C}[SL_2]$-comodules and the finite dimensional simple $\mathcal{U}(\mathfrak{sl}_2)$-modules \cite{kasselquantum}. However, we also show that they are not regular invertible as $\mathcal{U}(\mathfrak{sl}_2)$-modules. 

\begin{example}
Denote by $V_m$ the unique simple comodule over $\mathbb{C}[SL_2]$ of dimension $m$. By the Clebsch-Gordon formula, we have that for $m \leq n$
$$ V_m \otimes V_n = V_{n+m} \oplus V_{n+m-2} \oplus \ldots \oplus V_{n-m} $$
We can generalise this formula to
$$ V_m \bigotimes \left( \bigoplus_{n=m}^{\infty} \alpha_n V_n \right) = \bigoplus_{n=m}^{\infty} \left(V_m \bigotimes \alpha_n V_n\right) = \bigoplus_{n=m}^\infty \alpha_n \left( V_{n+m} \oplus V_{n+m-2} \oplus \ldots \oplus V_{m-n} \right) $$
We can therefore find a sequence $\{\alpha_n\}_{n=m}^\infty$ such that the right-hand side is the regular comodule. So $\mathcal{K}(V_m) = 1$ and we conclude that all simple comodules of $\mathbb{C}[SL_2]$ have trivial Knutson. In other words, we have shown that $\mathbb{C}[SL_2]$ is of "Knutson cotype".

We now show that if we consider $V_m$ as a $\mathcal{U}(\mathfrak{sl}_2)$-module, then $\mathcal{K}(V_m) = \text{dim}(V_m) = m$. By the previous proposition, we have that $\mathcal{K}(V_m) \leq m$. Suppose that 
$$ V_m \otimes M \cong k \ \mathcal{U}(\mathfrak{sl}_2) $$
where $k \leq m$. As $\mathcal{U}(\mathfrak{sl}_2)$ is finitely generated, there is a sufficiently large prime $p$ not dividing $m$ such that the previous isomorphism restricted to a field of characteristic $p$ yields
$$ V_m \otimes \overline{M} \cong k \ \mathcal{U}(\mathfrak{sl}_2) $$
Taking dimensions, we obtain that $m \cdot \text{dim}(\overline{M}) = k \cdot p^3$. So $m$ divides $k$ and we conclude that $\mathcal{K}(V_m) = m$.
\end{example}

\subsection{Group algebras and Knutson's Conjecture}

Let us start by giving the following reformulation of Knutson's Conjecture and revisit the smallest of the counter-examples presented by Savistkii to verify that the conjecture does not hold.

\begin{conjecture}[Knutson]
All complex finite dimensional group algebras are of Knutson type.
\end{conjecture}

\begin{example}
The group algebra $\mathds{C}[SL_2(\mathds{F}_5)]$ has nine simple modules. We claim that the two simple modules of dimension four are not regular invertible. This can be checked by computing their tensor product with all simple modules and showing that no integer linear combination of these products yields the regular module. Alternatively, we can show this computationally. See, for example, Code \ref{span} in the appendix. This code on Sage checks via Character Theory which irreducible representations of a group $G$ are regular invertible.
\end{example}

So having verified that Knutson's Conjecture does not hold as $\mathbb{C}[\text{SL}_2(\mathbb{F}_5)]$ is not of Knutson type, we are now interested in finding more counter-examples. As this is a group algebra of dimension $120$, we first compute on Sage with Code \ref{code index} the Knutson Index of all simple modules over group algebras of dimension up to $120$. It turns out that all group algebras up to this dimension are of Knutson type except eight counter-examples. These are $\mathbb{C}[{\rm D}_{12} \rtimes S_3]$, $\mathbb{C}[{\rm C}_3 \rtimes {\rm D}_{24}]$ and $\mathbb{C}[{\rm C}_3^2 \rtimes_2 Q_8]$ of dimension $72$ and $\mathbb{C}[{\rm C}_{15} \rtimes {\rm D}_8]$, $\mathbb{C}[{\rm C}_3 \rtimes {\rm D}_{40}]$, $\mathbb{C}[{\rm C}_5 \rtimes {\rm D}_{24}]$, $\mathbb{C}[{\rm C}_{15} \rtimes Q_8]$ and $\mathbb{C}[\text{SL}_2(\mathbb{F}_5)]$ of dimension $120$.

We then compute the Knutson Index of the potential counter-examples in Savitskii's Criterion for $q \leq 64$ and verify that they are all valid counter-examples to Knutson's Conjecture. We also compute the Knutson Index of the simple modules over some group algebras of larger dimensions and find further group algebras that are not of Knutson type. 

The following table contains these group algebras that are not of Knutson type, including the dimensions of the not regular invertible simple modules and their Knutson Indices. We include all counter-examples from Savitskii's Criterion as we will prove in Corollary \ref{cor Savitskii criterion}, at the end of this subsection, that the criterion holds. We can also check that Proposition \ref{prop divides} holds as the Knutson Index always divides the dimension of the respective module. \\

\begin{table}[h]
\begin{centering}
    \begin{tabular}{c|c|c|c}
        Group & Dimension of  & Dimensions of the simple and & Knutson \\[0pt] 
        algebra & the algebra & not regular invertible modules & Index \\[2pt] \hline
        $\mathbb{C}[{\rm SL}_2(\mathbb{F}_q)]$ & $(q-1) q (q+1)$ & $q-1$ or $q+1$ & $2$ \\[4pt]
        $\mathbb{C}[{\rm PSL}_2(\mathbb{F}_{q'})]$ & $\frac{(q-1) q (q+1)}{2}$ & $q-1$ or $q+1$ & $2$ \\[4pt]
        $\mathbb{C}[{\rm D}_{12} \rtimes S_3]$ & $72$ & $4$ & $2$  \\[4pt] 
        $\mathbb{C}[{\rm C}_3 \rtimes {\rm D}_{24}]$ & $72$ & $4$ & $2$ \\[4pt] 
        $\mathbb{C}[{\rm C}_3^2 \rtimes_2 Q_8]$ & $72$ & $4$ & $2$ \\[4pt] 
        $\mathbb{C}[{\rm C}_{15} \rtimes {\rm D}_8]$ & $120$ & $4$ & $2$ \\[4pt] 
        $\mathbb{C}[{\rm C}_3 \rtimes {\rm D}_{40}]$ & $120$ & $4$ & $2$ \\[4pt] 
        $\mathbb{C}[{\rm C}_5 \rtimes {\rm D}_{24}]$ & $120$ & $4$ & $2$ \\[4pt] 
        $\mathbb{C}[{\rm C}_{15} \rtimes Q_8]$ & $120$ & $4$ & $2$ \\[4pt] 
        $\mathbb{C}[A_{12}]$ & $\frac{12!}{2}$ & $3.696$ & $2$ \\[4pt]
        $\mathbb{C}[A_{13}]$ & $\frac{13!}{2}$ & $6.864$ & $2$ \\[4pt]
        $\mathbb{C}[A_{15}]$ & $\frac{15!}{2}$ & $32.032$ & $2$ \\[4pt]
        $\mathbb{C}[\text{Sz}(8)]$ & $29.120$ & $14$ & $2$ \\[4pt]
        $\mathbb{C}[SL_3(\mathbb{F}_5)]$ & $372.000$ & $124$ & $2$ \\[4pt]
        $\mathbb{C}[SL_3(\mathbb{F}_7)]$ & $5.630.688$ & $57, 288, 342, 399, 456$ & $3$ \\[4pt]
        $\mathbb{C}[SL_3(\mathbb{F}_9)]$ & $42.456.960$ & $728$ & $2$ \\[4pt]
        $\mathbb{C}[M_{22}]$ & $443.520$ & $21$, $210$, $231$ & $3$ \\[4pt]
        $\mathbb{C}[M_{23}]$ & $10.200.960$ & $231$ & $3$ \\[4pt]
        $\mathbb{C}[M_{24}]$ & $244.823.040$ & $252$, $5.544$, $5.796$ & $3$ \\
    \end{tabular} \\
\end{centering}
\ \\
where $q \geq 5$ is odd and $q'$ is odd and not plus or minus one a power of two.
    \caption{Knutson Indices of counter-examples to Knutson's Conjecture}
    \label{Knutson Index}
\end{table}

Note also that a direct product of a counter-example with any other groups yields further counter-examples by the following result.

\begin{proposition}
    Let $M$ be a $\mathbb{C}[G]$-module and $N$ a $\mathbb{C}[H]$-module. Their external tensor product $M \boxtimes N$ is $\mathbb{C}[G \times H]$-module and
    $$ \mathcal{K}(M), \mathcal{K}(N) \leq \mathcal{K}(M \boxtimes N) \leq \mathcal{K}(M) \cdot \mathcal{K}(N) $$
\end{proposition}

\begin{proof}
    We obtain the first inequality by restricting the $\mathcal{K}(M \boxtimes N)$-regular inverse of $M \boxtimes N$ to $\mathbb{C}[G]$ and $\mathbb{C}[H]$ respectively. For the second, note that the external tensor product of a $\mathcal{K}(M)$-regular inverse and $\mathcal{K}(N)$-regular inverse yields a $\mathcal{K}(M) \mathcal{K}(N)$-regular inverse for $M \boxtimes N$.
\end{proof}

Our computations show that $\mathcal{K}(M \boxtimes N) = \mathcal{K}(M) \cdot \mathcal{K}(M)$ in all cases considered. For instance, $\mathbb{C}[SL_2(\mathbb{F}_5) \times M_{22}]$ has simple module of dimensions $84, 840$ and $924$ with Knutson Indices six. So we propose the following question.

\begin{question}
    Does it hold in general that $\mathcal{K}(M \boxtimes N) = \mathcal{K}(M) \cdot \mathcal{K}(M)$?
\end{question}

It is also possible to show that many families of group algebras are of Knutson type. For instance, let us consider the following examples.

\begin{example}
All finite dimensional commutative group algebras are of Knutson type.
\end{example}

\begin{example} \label{ex extraordinary}
All group algebra of an extraspecial group are of Knutson type.
\end{example} 

\begin{proof}
We can deduce from character tables of the extraspecial groups \cite{gorenstein} that their representation ring has the following structure. There are $p^{2n}$ one-dimensional modules $V_1, \ldots V_{p^{2n}}$ and $p-1$ $p^n$-dimensional simple modules $W_1, \ldots W_{p-1}$ with
$$ W_i \otimes V_j \cong W_i \quad \text{ and } \quad W_i \otimes W_j \cong p^n W_{k} $$
where $k \equiv i + j$ mod $p$ and $p^n W_0$ is the sum of all one-dimensional modules.

The Knutson Index of all one-dimensional modules is clearly one and for the $p^n$-dimensional ones, we have that 
$$ W_i \otimes (p^n V_1 \oplus W_1 \oplus \ldots \oplus W_{p-1}) \cong V_1 \oplus \ldots \oplus V_{p^{2n}} \oplus p^n (W_{1} \oplus \ldots \oplus W_{p-1}) $$
So we conclude that all simple modules of the group algebra of an extraspecial group have trivial Knutson Index and that these group algebras are of Knutson type.
\end{proof} 

\begin{example} \label{ex dih}
All group algebras of a dihedral group $\mathbb{C}[D_{2n}]$ are of Knutson type.
\end{example}

\begin{proof}
The character table of $D_{2n}$ with $n$ even is given by

\begin{table} [h]
$$
\begin{array}{c|ccc}
  \rm class & 1 & r^k & s r^k \cr
\hline
  \chi_1 & 1& 1 & 1 \cr
  \chi_2 & 1 & 1 & -1 \cr
  \chi_3 & 1 & (-1)^k & (-1)^k \cr
  \chi_4 & 1 & (-1)^k & (-1)^{k+1} \cr
  \rho_h & 2 & 2 \text{cos} ( \frac{2 h k \pi}{n} ) & 0
\end{array}
$$
where $1 \leq h \leq \frac{n}{2}-1$.
\caption{Character table of $D_{2n}$ for $n$ even \cite{serre1971representation}} 
\end{table}

We deduce that there are four one-dimensional modules $V_1, V_2, V_3, V_4$ and $\tfrac{n}{2} - 1$ two-dimensional simple modules $W_1, \ldots W_{\tfrac{n}{2}-1}$, where $V_1$ and $V_3$ act trivially on the set of two-dimensional modules and $V_2$ and $V_4$ act by mapping $W_h$ to $W_{\tfrac{n}{2} - h}$. In particular, all $V_i$'s act trivially on $W_{\tfrac{n}{4}}$. Let 
$$ V = \bigoplus_{i=1}^4 V_i \qquad \qquad W = \bigoplus_{i=1}^{\tfrac{n}{2}-1} W_i $$
If we cancel the previous from $W_h \otimes H \cong 2H$ and divide by two, we obtain that for $k \neq \tfrac{n}{4}$
$$ W_k \otimes W \cong V \oplus  W_k \oplus W_{\tfrac{n}{2}-k} \oplus_{i \in I} 2 W_i \qquad \text{and} \qquad W_{\tfrac{n}{4}} \otimes W \cong V \oplus_{j \in J} 2 W_j $$
where $I = \{1, \ldots, \tfrac{n}{2}-1\}^{\setminus \{k, \frac{n}{2}-k\}}$ and $J = \{1, \ldots, \tfrac{n}{2}-1\}^{\setminus \{\frac{n}{4}\}}$ 

So $V_1 \oplus V_2 \oplus W_1 \oplus \ldots \oplus W_{\tfrac{n}{2}-1}$ is a regular inverse for any $W_h$.

The case with $n$ odd is similar, but with only two one-dimensional modules.
\end{proof}

Our computations also show that the group algebras of a symmetric group $\mathbb{C}[S_n]$ for $n \leq 15$ are of Knutson type. The following question remains open.

\begin{question}
    Are all group algebras of a symmetric group $\mathbb{C}[S_n]$ of Knutson type?
\end{question}

We now introduce two algorithms that, as we show later, construct a $\mathcal{K}(M)$-regular inverse for a simple $\mathbb{C}[G]$-module $M$ via induction from the group algebras of the Sylow $p$-groups of $G$. This idea is similar to the one introduced by Savitskii \cite{savitskii} but for modules instead of for characters and with a different application. These algorithms are useful tools to compute the Knutson Indices of simple modules over more complicated group algebras $\mathbb{C}[G]$ as they reduce the problem to computing the Knutson Indices over the group algebras of the Sylow $p$-subgroups of $G$. 

\begin{definition}[Savitskii Restriction] \label{Savitskii Restriction}
Let $\mathbb{C}[G]$ be a group algebra where $G$ is a group of order $p_1^{k_1} \cdot \ldots \cdot p_m^{k_m}$ and $S_1, \ldots, S_m$ its Sylow $p$-subgroups. Let $M$ be an $r$-dimensional module over $\mathbb{C}[G]$ and define the Savitskii Restriction $M_i$ to each group algebra $\mathbb{C}[S_i]$ using the following algorithm.

Let $p_i^{l_i}$ be the greatest common divisor of $r$ and $p_i^{k_i}$. Then by Bezout's identity there exists integers $u_i$ and $v_i$ such that $u_i r + v_i p_i^{k_i} = p_i^{l_i}$. For simplicity, we will always choose $u_i$ minimal positive. The Savitskii Restriction of $M$ to $\mathbb{C}[S_i]$ is 
$$ M_i = u_i M|_{S_i} + v_i V_{S_i} $$
where  $M|_{S_i}$ is the usual restriction of $M$ to $\mathbb{C}[S_i]$ and $V_{S_i}$ is the regular module of $\mathbb{C}[S_i]$.
\end{definition}

\begin{definition} [Savitskii Induced Inverse] \label{savitskii induction}
Let $\mathbb{C}[G]$ be a group algebra where $G$ is a group of order $p_1^{k_1} \cdot \ldots \cdot p_m^{k_m}$ and $S_1, \ldots, S_m$ its Sylow $p$-subgroups. Let $M$ be an $r$-dimensional module over $\mathbb{C}[G]$ and $M_i$ its Savitskii Restriction to $\mathbb{C}[S_i]$. Denote by $N_i$ a $\mathcal{K}(M_i)$-regular inverse for $M_i$ and by $N_i^*$ be the induced module of $N_i$ to $\mathbb{C}[G]$. Then $\text{gcd}(\text{dim}(N_1^*), \ldots, \text{dim}(N_m^*)) = p_1^{l_1} \cdot \ldots \cdot p_m^{l_m}$ where $l_i \leq k_i$. The Savitskii Induced Inverse of $M$ is
$$ N = a_1 N_1^* + \ldots + a_m N_m^* $$
where $a_i \in \mathbb{Z}$ such that $p_1^{l_1} \cdot \ldots \cdot p_m^{l_m} = a_1 N_1^* + \ldots + a_m N_m^*$.
\end{definition}

The following result shows the Savitskii Induced Inverse of $M$ is a $\mathcal{K}(M)$-regular inverse of $M$ and that the Knutson Index of a module over a group algebra is the product of the Knutson Indices of its Savitskii Restrictions to the group algebras of all Sylow $p$-subgroups.

\begin{theorem} \label{thm Knutson Index vs reduction}
Let $M$ be a finite dimensional module over a finite dimensional group algebra $\mathbb{C}[G]$ and $N$ be the Savitskii Induced Inverse of $M$. Then 
$$ M \otimes N \cong \mathcal{K}(M) \ G \quad \text{and} \quad \mathcal{K}(M) = \prod_{i=1}^m \mathcal{K}(M_i) $$
where $M_1, \ldots, M_m$ are the Savitskii Restrictions of $M$.
\end{theorem}

\begin{proof}
Let $\chi$ and $\mu$ be the characters corresponding to $M$ and its Savitskii Induced Inverse respectively and $\chi_1, \ldots \chi_m$ the characters of the Savitskii Restrictions. We first claim that
$$ \chi \otimes \mu = \prod_{i = 1}^k \mathcal{K}(\chi_i) \ \rho_\text{reg} $$

The degree of an irreducible character divides the order of the group and $\mathcal{K}(\chi)$ divides the degree of the character by Proposition \ref{prop divides}. So we have that
$$ |G|= \prod_{i=1}^{k} p_i^{l_i} \qquad \qquad \chi(\text{id}) = \prod_{i=1}^{k} p_i^{m_i} \qquad \qquad \mathcal{K}(\chi) = \prod_{i=1}^{k} p_i^{n_i} $$
where $n_i \leq m_i \leq l_i$. It follows from the definitions of Sylow $p$-subgroups and Savitskii Restrictions that
$$ |S_i| = p_i^{l_i} \qquad \qquad \chi_i(\text{id}) = p_i^{m_i} $$
Let $\mu_i$ be a $\mathcal{K}(\chi_i)$-regular inverse of $S_i$. Then $\mu_i(\text{id}) = p_i^{r_i}$ where $r_i = n_i + l_i - m_i \leq l_i$. So
$$ \mu_i^*(\text{id}) = p_i^{r_i} \cdot \prod_{\begin{subarray}{c} j = 1 \\ j \neq i \end{subarray}}^{k} p_j^{l_j} \quad \implies \quad \text{gcd}(\mu_1^*(\text{id}), \ldots \mu_k^*(\text{id})) = \prod_{i = 1}^k p_i^{r_i} $$ 
Let $a_i \in \mathbb{Z}$ such that 
$$ \prod_{i=1}^k p_i^{r_i} = \sum_{i=1}^k a_i \mu_i^*(\text{id}) \quad \text{and} \quad \mu = \sum_{i=1}^k a_i \mu_i^* \ \implies \ \mu(\text{id}) = \prod_{i=1}^k p_i^{r_i} $$
So we have that
$$ (\chi \otimes \mu)(\text{id}) = \chi(\text{id}) \mu(\text{id}) = \prod_{i=1}^k p_i^{m_i} \cdot \prod_{i=1}^k p_i^{r_i} = \prod_{i=1}^k p_i^{n_i + l_i} = \prod_{i=1}^k \mathcal{K}(\chi_i) \ |G| $$
It remains to show that $(\chi \otimes \mu)(g) = 0 \ \forall g \in G \setminus \{\text{id}\}$. We show that if $\chi(g) \neq 0$ for $g \neq (\text{id})$ then $\mu(g) = 0$. If the order of $g$ is not a prime power then $\mu_i^*(g) = 0$ for all $i$ by definition of usual induction of characters and therefore $\mu(g) = 0$. If the order of $g$ is $p_i^b$ then $\mu_j^*(g) = 0$ for $j \neq i$ so it suffices to show that $\mu_i^*(g) = 0$. By Mackey's formula, we have that
$$ \mu_i^*(g) = \sum_{\begin{subarray}{c} S_i h \in G \setminus S_i \\ h g h^{-1} \in S_i \end{subarray}} \mu_i (h g h^{-1}) $$
As $\chi(g) \neq 0$ we have that $\forall h \in G$,
$$ \chi(h g h^{-1}) \neq 0 \quad \implies \quad \chi_i(h g h^{-1}) \neq 0 \quad \implies \quad \mu_i(h g h^{-1}) = 0 $$
where the last implication follows from $\mu_i \otimes \chi_i = \mathcal{K}(\chi_i) \ \rho_{\text{reg}} $. Hence $\mu_i^*(g) = 0$ and we have shown that 
$$ \chi \otimes \mu = \prod_{i = 1}^k \mathcal{K}(\chi_i) \ \rho_\text{reg} $$
So in particular we have that
$$ \mathcal{K}(\chi) \leq \prod_{i = 1}^k \mathcal{K}(\chi_i) $$
If the inequality is strict then
$$ \mathcal{K}(\chi) = \prod_{i=1}^k p_i^{n_i} < \prod_{i=1}^k p_i^{r_i} = \prod_{i=1}^k \mathcal{K}(\chi_i) \quad \implies \quad n_i < r_i \text{ for some } i $$
So restricting to $\mathbb{C}[S_i]$ yields that $\chi_i \otimes \mu_i = p_i^{n_i} \ \rho_\text{reg}$ and this contradicts the assumptions $\mathcal{K}(\chi_i) = p_i^{r_i}$. So we conclude that the Knutson Index of an irreducible character is the product of the Knutson Indices of its Savitskii Restrictions to all Sylow $p$-subgroups.
\end{proof}

From this theorem, we obtain the following three corollaries, including Savitskii's Criterion. 

\begin{corollary} \label{cor maximal}
Let $\mathbb{C}[G]$ be a finite dimensional group and $p^a$ the highest prime power that divides the dimension of a simple $\mathbb{C}[G]$-module $M$. If $p^a$ is also the highest power of $p$ that divides the order of the group $G$, then $p$ does not divide the Knutson Index of $M$.
\end{corollary}

\begin{proof}
The Savitskii Restriction of $M$ to the group algebra $\mathbb{C}[S]$, where $S$ is a Sylow $p$-subgroup of $G$, yields the regular module of $\mathbb{C}[S]$ and has therefore trivial Knutson Index. For all other restrictions, $p$ does not divide the order of the Sylow subgroups, so by Proposition \ref{prop divides}, we know that $p$ does not divide their Knutson Indices. We conclude by the previous theorem that the Knutson Index of $M$ is not divisible by $p$.
\end{proof}

\begin{corollary} \label{cor distinct primes}
If the order of a group $G$ is a product of distinct primes, then the group algebra $\mathbb{C}[G]$ is of Knutson type. 
\end{corollary}

\begin{proof}
Let $M$ be a simple $\mathbb{C}[G]$-module and $S$ is a Sylow $p$-subgroup of $G$. The Savitskii Restriction of $M$ to $\mathbb{C}[S]$ is the regular module if $p$ divides $\text{dim}(M)$ and a one-dimensional module otherwise. In both cases, the Knutson Index is trivial and it follows from Theorem \ref{thm Knutson Index vs reduction} that $\mathcal{K}(M) = 1$.
\end{proof}

\begin{corollary} [Savitskii's Criterion] \label{cor Savitskii criterion}
\ \\
The group algebra $\mathbb{C}[{\rm SL}_2(\mathbb{F}_q)]$ is of Knutson type if and only if $q=2^m$ or $q=3$. \\
The group algebra ${\mathbb{C}[\rm PSL}_2(\mathbb{F}_q)]$ is of Knutson type if and only if $q=2^m$ or $q=2^m \pm 1$.
\end{corollary}

\begin{proof}
For $\mathbb{C}[{\rm SL}_2(\mathbb{F}_3)]$, it is straightforward to construct a regular inverse for all simple modules. Note that if $q=2^m$ then $\mathbb{C}[{\rm SL}_2(\mathbb{F}_q)] = \mathbb{C}[{\rm PSL}_2(\mathbb{F}_q)]$ and we will show now that all their simple modules are regular invertible. The dimension of the group algebra is $(q-1) q (q+1)$ and of the simple modules $1$, $q-1$, $q$ and $q+1$ \cite{dornhoff1971group}. As no odd prime divides both $q-1$ and $q+1$, we see that any odd prime dividing the dimension of the module is maximal, so no odd prime divides their Knutson Indices by Corollary \ref{cor maximal}. As $q-1$ and $q+1$ are odd, the modules of these dimensions must have trivial Knutson Index by Proposition \ref{prop divides}. Also, $q$ is the highest power of $2$ dividing the order of the dimension of the algebra. Therefore the $q$-dimensional module also has trivial Knutson Index by Corollary \ref{cor maximal}. Hence all simple modules have trivial Knutson Index. 

Let us now show that $\mathbb{C}[{\rm SL}_2(\mathbb{F}_q)]$ with $q \geq 5$ odd is not of Knutson type. The following table contains the dimensions and possible Knutson Indices of its simple modules depending on the congruence class of $q$ modulo $4$.

\newpage
\begin{table}[h]
$$
\begin{array}{c|c|c}
  \rm Dimension & \text{Knutson Index if } q \equiv 1 \text{ mod } 4 & \text{Knutson Index if } q \equiv 3 \text{ mod } 4 \cr
\hline
  1 & 1 & 1 \cr
  \frac{1}{2}(q-1) & 1 & 1 \cr
  \frac{1}{2}(q+1) & 1 & 1 \cr
  q - 1 & 2 & 1 \text{ or } 2 \cr
  q & 1 & 1 \cr
  q + 1 & 1 \text{ or } 2 & 2 \cr
\end{array}
$$
\caption{Knutson Indices of simple modules of $\mathbb{C}[{\rm SL}_2(\mathbb{F}_q)]$ with $q \geq 5$ odd}
\end{table}

We can obtain these indices in the following way. The dimension of the group algebra is again $(q+1) q (q-1)$ and the simple modules are of dimension $1$, $\tfrac{q-1}{2}$, $\tfrac{q+1}{2}$, $q-1$, $q$ or $q+1$ \cite{dornhoff1971group}. By similar arguments as for $q$ even, no odd prime divides their Knutson Index. So the odd-dimensional modules have trivial Knutson Index and the Knutson Index of the even-dimensional modules is a power of $2$. So by Theorem \ref{thm Knutson Index vs reduction}, the Knutson Index of a simple module over $\mathbb{C}[{\rm SL}_2(\mathbb{F}_q)]$ equals the Knutson Index of its Savitskii Restriction to the Sylow $2$-subgroup. 

To compute the Knutson Index of the remaining modules we will use Savitskii Restriction to the Sylow $2$-subgroup of ${\rm SL}_2(\mathbb{F}_q)$ for odd $q$. This is a generalised quaternion group \cite{huppert} and it can be shown via Clifford Theory that $\mathbb{C}[Q_{2^k}]$ and $\mathbb{C}[D_{2^k}]$ have isomorphic representation rings \cite{feit}. So it follows from the character table of $D_{2n}$ in Example \ref{ex dih} that $\mathbb{C}[Q_{2^k}]$ has four one-dimensional modules and $2^{k-2}-1$ two-dimensional simple modules. As the index depends on the congruence of $q$ modulo $4$, we will only consider the case $q \equiv 1$ mod $4$ here. The other case is analogous, with the $q-1$ and $q+1$ dimensional modules swapping roles. 

If $q \equiv 1$ mod $4$ then $\frac{1}{2}(q+1)$ is odd-dimensional so the modules of that dimension have trivial Knutson Index. Let us now consider the $\frac{1}{2}(q-1)$ dimensional modules. If $2^k$ is the highest power of $2$ diving the order of the group then ${\rm Q}_{2^k}$ is the Sylow $2$-subgroup of ${\rm SL}_2(\mathbb{F}_q)$ and the highest power diving of $2$ dividing $\frac{q-1}{2}$ is $2^{k-2}$. So applying Savitskii Restriction yields a $2^{k-2}$ dimensional module. This module acts traceless on all elements except on the identity and the unique element of order $2$. So to find a regular inverse in ${\rm Q}_{2^k}$ we just need a $4$-dimensional module acting traceless on this element order $2$. In the isomorphism between the character tables of $D_{2^k}$ and $Q_{2^k}$ this element corresponds to the rotation of order two. We see from the character table of $D_{2n}$ with $n$ even in Example \ref{ex dih} that the irreducible character of degree one evaluated at this element are $1$ and the ones of degree two are $\pm 2$. So it follows that there is a $4$-dimensional $\mathbb{C}[Q_{2^k}]$-module acting traceless on the element of order $2$.

Next, we show that all $q-1$ dimensional modules have Knutson Index $2$. Savitskii Restriction to $Q_{2^k}$ gives us a module of dimension $2^{k-1}$ and this acts again traceless on all elements except the identity and the unique element of order $2$. Using the same regular inverse as before, we see that Knutson Index is at most $2$. We know that the Knutson Index is non-trivial because no $2$-dimensional virtual module of $\mathbb{C}[Q_{2^k}]$ acts traceless on the element of order $2$.

The remaining modules are of dimension $q+1$. As $q + 1 \equiv 2$ mod $4$ the Savitskii restriction to the Sylow $2$-subgroup is two-dimensional and its Knutson Index $1$ or $2$. As we will show later for $PSL_2(q)$, if $q = 2^m +1$ then all have Knutson Index one and if $q \neq 2^m +1$ then same have Knutson Index one and others two.

Let us now compute the Knutson Indices of the simple modules over $\mathbb{C}[{\rm PSL}_2(\mathbb{F}_q)]$ for odd $q$. We obtain the following table. 

\begin{table} [h]
$$
\begin{array}{c|c|c}
  \rm Dimension & \text{Knutson Index if } q \equiv 1 \text{ mod } 4 & \text{Knutson Index if } q \equiv 3 \text{ mod } 4 \cr
\hline
  1 & 1 & 1 \cr
  \frac{1}{2}(q-1) & - & 1 \cr
  \frac{1}{2}(q+1) & 1 & - \cr
  q - 1 & 1 & 1 \text{ or } 2 \cr
  q & 1 & 1 \cr
  q + 1 & 1 \text{ or } 2 & 1 \cr
\end{array}
$$
\caption{Knutson Indices of simple modules of $\mathbb{C}[{\rm PSL}_2(\mathbb{F}_q)]$ with $q \geq 5$ odd}
\end{table}

For the modules of dimension $1, \frac{q-1}{2}, \frac{q+1}{2}, q$ we obtain that the Knutson Index is one because they are of odd dimension. For $q \equiv 1$ mod $4$, note that the highest power of two dividing $q-1$ equals the highest power of two dividing the order of $PSL_2(q)$ so by Corollary \ref{cor maximal} we have that the Knutson Index is one. Similarly, if $q \equiv 3$ mod $4$, then the Knutson Index of the modules of dimension $q+1$ is one.

We now show that if $q \equiv 1$ mod $4$, then if $q = 2^m + 1$, all $q+1$ dimensional modules have trivial Knutson Index and if $q \neq 2^m+1$ then some will have Knutson Index one and others two. The Savitskii Restrictions of these modules yield a two-dimensional module of $\mathbb{C}[D_{2^k}]$ and these modules are regular invertible if and only if they act traceless on at least one element. To see this, note that if it does not act traceless on any element then the regular inverse must act traceless on all non-trivial elements. So it must be an integer multiple of the regular modules and can not be of the required dimension. Conversely, suppose that it acts traceless on an element $g$. We can then construct a regular inverse for that module by taking the direct sum of all simple modules acting with positive trace on $g$ times their dimension and all simple modules acting traceless on $g$ times half their dimension. We conclude by noticing that in the character table of $SL_2(q)$ all $q+1$ dimensional representations act traceless on a $2$-element if and only $q = 2^m + 1$. If $q \equiv 3$ mod $4$ then the same arguments work for the modules of dimension $q-1$.
\end{proof}

\subsection{Semisimple Hopf algebras}

We now consider the more general case where $H$ is a finite dimensional semisimple Hopf algebra over a fixed algebraically closed field $k$ of characteristic zero. Let us start by recalling some results about these Hopf algebras and considering some examples.

\begin{proposition} \label{prop fd Hopf algebra} \
\begin{enumerate} [a)]
    \item $H$ is cocommutative if and only if $H \cong k[G]$ for a finite group $G$.
    \item $H$ is commutative if and only if $H \cong k[G]^*$ for a finite group $G$.
    \item $H$ is commutative and cocommutative if and only if $H \cong k[G]$ for an abelian finite group $G$.
\end{enumerate}
\end{proposition}

As all finite dimensional commutative Hopf algebras are of Knutson type, we will be mainly interested in the non-commutative case. By the previous proposition, this means that we only need to consider group algebras of non-abelian groups and non-commutative and non-cocommutative Hopf algebras. It follows by our computations in the previous subsection that a semisimple Hopf algebra of dimension up to $72$ that is not Knutson type must be a non-commutative and non-cocommutative Hopf algebra.

It is known that semisimple Hopf algebras of dimension $p$ \cite{zhup} or $p^2$ \cite{masuokap2} where $p$ is prime are commutative group algebras. For dimension $p q$ where $p$ and $q$ are distinct primes, we have that it is a group algebra or the dual of a group algebra \cite{etingofpq}, where the former case is of Knutson type by Corollary \ref{cor distinct primes}.

\begin{example} \label{ex p p2 pq}
All semisimple Hopf algebras over $k$ of dimension $p$, $p^2$ or $p q$ where $p$ and $q$ are distinct primes are of Knutson type.
\end{example} 

Recall that a non-zero element $g \in H$ is called group-like if $\Delta(g) = g \otimes g$. The set of group-like elements $G(H)$ forms a finite group and its order divides the dimension of $H$ by the Nichols-Zoeller Theorem \cite{NZ}. Furthermore, $G(H^{*})$ is isomorphic to the group of one-dimensional modules of $H$ under the tensor product and acts on the set of simple modules of a fixed dimension. 

\begin{lemma}
If $W_1$ and $W_2$ are simple modules over $H$ then a one-dimensional module $V_g$ appears in the direct sum decomposition of $W_1 \otimes W_2$ at most once and if and only if $W_1^* \cong W_2 \otimes V_{g^{-1}}$.
\end{lemma}

\begin{proof}
This follows from the fact that 
$$ \rm Hom_k(W_1^*,W_2 \otimes V_{g^{-1}}) \cong W_1 \otimes W_2 \otimes V_{g^{-1}} $$
The number of linearly independent isomorphisms between $W_1^*$ and $W_2 \otimes V_{g^{-1}}$ corresponds to the multiplicity of $V_{\rm id}$ in the direct sum decomposition of $W_1 \otimes W_2 \otimes V_{g^{-1}}$ and therefore to the multiplicity of $V_g$ in the direct sum decomposition of $W_1 \otimes W_2$. By Schur's Lemma, all isomorphisms are scalar multiples and the result follows.
\end{proof}

If a Hopf algebra $H$ has precisely $a_i$ simple modules of dimension $n_i$, then we say that the module structure of $H$ is $(n_1^{a_i}, \ldots, n_r^{a_r})$.

\begin{corollary} \label{cor unique higher dimension}
A Hopf algebra $H$ with module structure $(1^a, n)$ is of Knutson type.
\end{corollary}

\begin{proof}
Let $W$ be the unique module of dimension $n$, $V_g$ any one-dimensional module and $V$ the direct sum of all one-dimensional modules. As $W \otimes V_g$ is a simple module of dimension $n$, we must have that $W \otimes V_g \cong W$. Similarly, $W^* \cong W$ and this implies that $W^* \cong W \otimes V_{g^{-1}}$.

It follows from the previous lemma that all one-dimensional modules appear precisely ones in the direct sum decomposition of $W \otimes W$. So $W \otimes W \cong V \otimes s W$ where $s = \frac{n^2 - a}{n}$. Hence
$$ W \otimes \left(\dfrac{a}{n} V_{\text{id}} \oplus W \right) = H $$
and we conclude that $H$ is of Knutson type.
\end{proof}

\begin{proposition} \label{prop 8}
All semisimple Hopf algebras of dimension $8$ are of Knutson type. 
\end{proposition}

\begin{proof}
If $H$ is non-commutative then its module structure is $(1^4, 2)$. So by Corollary \ref{cor unique higher dimension}, it is of Knutson type.
\end{proof}

The smallest dimension in which a non-commutative and non-cocommutative semisimple Hopf algebra appears is $8$ and is known as the Kac-Paljutkin algebra. So we have shown that all semisimple Hopf algebras of dimension up to $8$ are of Knutson type. The next non-commutative and non-cocommutative semisimple Hopf algebra shows up in dimension $12$.

\begin{proposition} \label{prop 12}
All semisimple Hopf algebras over $k$ of dimension $12$ are of Knutson type.
\end{proposition}

\begin{proof}
There are five non-commutative Hopf algebras of dimension $12$. These are two non-commutative and non-cocommutative Hopf algebras and the group algebras of the three non-abelian groups of order $12$, the dicyclic group $\rm Dic_3$, the dihedral group $\rm D_{12}$ and the alternating group $A_4$. The latter has module structure $(1^3, 3)$ so it follows by Corollary \ref{cor unique higher dimension} that it is of Knutson type. The other four Hopf algebras have module structure $(1^4,2^2)$ with two one-dimensional modules acting trivially on the two-dimensional modules and the other two one-dimensional modules acting by swapping them \cite{fukuda12}. We can therefore discuss them together.

Denote by $V_1, V_2, V_3, V_4$ the one-dimensional modules and by $W_1, W_2$ the simple two-dimensional modules, where $V_1$ and $V_3$ are the ones acting trivially on them.

So we have that $W_i \otimes (V_1 \oplus V_2 \oplus V_3 \oplus V_4) \cong 2 W_1 \oplus 2 W_2$. After cancelling this from $W_i \otimes H \cong 2 H$ and dividing by two, we obtain the following.
$$ W_i \otimes (W_1 \oplus W_2) \cong V_1 \oplus V_2 \oplus V_3 \oplus V_4 \oplus W_1 \oplus W_2 $$
This show that $\mathcal{K}(W_1) = \mathcal{K}(W_2) = 1$ as
$$ W_i \otimes (V_1 \oplus V_2 \oplus W_1 \oplus W_2) \cong H $$
So all semisimple Hopf algebras of dimension $12$ are of Knutson type.
\end{proof}

Note that we could construct the regular inverses for the $W_i$'s straightforwardly because of the following two facts. First, all simple modules of $H$ are either one-dimensional or $n$-dimensional for some fixed $n > 1$. Secondly, we could divide by $n$ as the order of the stabiliser of all $W_i$'s under the action of the one-dimensional modules is divisible by $n$. In fact, we now show if the first assumption holds, then the second must hold and the Hopf algebra must be of Knutson type.

\begin{lemma} \label{lemma divides stab}
Let $W$ be a simple $H$-module. If the dimension of $W$ divides the dimensions of all simple modules of dimension greater than one, then the dimension of $W$ also divides the order of stabiliser of $W$ under the action of the $V_g$'s.
\end{lemma}

\begin{proof}
Let $\text{dim}(W)=n$ and denote by $V^{(i)}$ be the direct of all modules $V_g$ such that $W \otimes V_g \cong W$. Then
$$ W \otimes W^* = V^{(i)} \oplus W' $$
where the dimension of $W'$ is divisible by $n$. So $n^2 = |\text{Stab}(W_i)| + k n$ and $n$ divides $|\text{Stab}(W_i)|$.
\end{proof}

\begin{theorem} \label{thm two dim}
A finite dimensional semisimple Hopf algebra $H$ over $k$ with module structure $(1^a, n^b)$ is of Knutson type.
\end{theorem}

\begin{proof}
Fix $W_i$ and let
$$ V = \bigoplus_{j=1}^k V_j \qquad \qquad W = \bigoplus_{j=1}^l W_j \qquad \qquad W^{(i)} = \bigoplus_{j \in J_i} W_j $$
where $J_i = \{j \in [1,l] \ | \ W_j \in \text{Orb}(W_i) \}$. We have that
$$ W_i \otimes V \cong |\text{Stab}(W_i)| \ W^{(i)} $$
So substituting this into $ W_i \otimes H \cong W_i \otimes (V \oplus n W) \cong n H $
yields 
$$ |\text{Stab}(W_i)| \ W^{(i)} \oplus (W_i \otimes n W) \cong n H $$
By Lemma \ref{lemma divides stab} we have that $n$ divides $|\text{Stab}(W_i)|$ so
$$ k  W^{(i)} \oplus (W_i \otimes W) \cong H $$
where $k = \frac{|\text{Stab}(W_i)|}{n}$. We conclude that
$$ W_i \otimes (k V^{(i)} \oplus W) \cong H $$
where $V^{(i)}$ as a direct sum of one-dimensional modules such that $W_i \otimes V^{(i)} \cong W^{(i)}$.
\end{proof}

\begin{remark} \label{rmk Kaplansky}
Let $p, q$ and $r$ be distinct primes. Kaplansky's Sixth Conjecture holds for semisimple Hopf algebras of dimensions $p$, $p^2$ and $p q$ by their classifications. The conjecture is also true for dimensions $p q^2$, $p q r$ \cite{etingofKaplansky} and $p^n$ \cite{montgomery-witherspoon}. So we will assume that in these cases, the dimension of a simple $H$-module divides the dimension of the Hopf algebra $H$.
\end{remark}

\begin{corollary} \label{cor p3 p4}
All semisimple Hopf algebras of dimension $p^3$ and $p^4$ are of Knutson type. 
\end{corollary}

\begin{proof}
A simple module over Hopf algebra of dimension $p^3$ or $p^4$ must be of dimension $1$ or $p$. So by Theorem \ref{thm two dim}, we conclude that all semisimple Hopf algebras of dimension $p^3$ and $p^4$ are of Knutson type.
\end{proof}

We can further generalise the previous theorem by dropping the assumption that all simple modules are of two possible dimensions.

\begin{theorem} \label{thm dim divides}
Let $U$ be a simple module over a finite dimensional semisimple Hopf algebra $H$. If $dim(U)$ divides the order of all simple modules that are not one-dimensional, then $\mathcal{K}(U) = 1$.
\end{theorem}

\begin{proof}
Let $n = \text{dim}(U)$ and denote by $V$ be the direct sum of all one-dimensional modules and by $W$ such that $H \cong V \oplus W$. As $n$ divides the dimension of every module appearing in $W$, we can write $W \cong n W'$ for some module $W'$. Denote by $U'$ be the direct sum of all modules in the orbit of $U$. As $U \otimes H \cong U \otimes (V \oplus W) \cong n H$ we have that 
$$ |\text{Stab}(U)| \ U' \oplus (U \otimes W) \cong n H $$
by our assumptions, this is divisible by $n$, so 
$$ U \otimes (k V' \oplus W') \cong H $$
where $k = \frac{|\text{Stab}(U)|}{n}$ and $V'$ 
\end{proof}

Putting all our previous results together and considering some special cases individually, we obtain the following corollary.

\begin{corollary}
All semisimple Hopf algebras over $k$ of dimension up to $31$, except the Hopf algebras of dimension $24$ with module structure $(1^2, 2, 3^2)$ type II, are of Knutson type. 
\end{corollary}

\begin{proof}
First, recall that semisimple Hopf algebras of dimensions $p$, $p^2$ or $p q$, where $p$ and $q$ are distinct primes, are of Knutson type by Example \ref{ex p p2 pq}. We also know this is the case for semisimple Hopf algebras of dimension $p^3$ and $p^4$ by Corollary \ref{cor p3 p4}. Therefore, we just need to prove that a Hopf algebras $H$ that does not fall into any of these categories is of Knutson type.

\begin{itemize}
    \item \textbf{Dimension $\mathbf{12}$}: See Proposition \ref{prop 12}.
    \item \textbf{Dimension $\mathbf{18}$}: The order of $G(H^*)$ is $9, 6, 2$ \cite{masuoka18}. It follows that the module structure must be $(1^9, 3)$, $(1^6, 2^3)$ or $(1^2, 2^4)$ and by Theorem \ref{thm two dim} the Hopf algebra is of Knutson type. 
    \item \textbf{Dimension $\mathbf{20}$}: The module structure of $H$ is $(1^4, 2^4)$ or $(1^4, 4)$ and it follows from Theorem \ref{thm two dim} that in both cases the Hopf algebra is of Knutson type.
    \item \textbf{Dimension $\mathbf{24}$}: If $H$ is non-commutative then the module structure must be $(1^{12}, 2^3)$, $(1^8, 2^4)$, $(1^8, 4^1)$, $(1^6, 3^2)$, $(1^4, 2^5)$, $(1^3, 2^3, 3)$ or $(1^2, 2, 3^2)$ \cite{natale24-30-40-42-54-56}. By Theorem \ref{thm two dim} we just need to consider the cases $(1^3, 2^3, 3)$ and  $(1^2, 2, 3^2)$.
    
    \underline{Case $(1^3, 2^3, 3)$}: Let $V_1, V_2, V_3$ be the one-dimensional modules, $W_1, W_2, W_3$ the simple two-dimensional modules and $U$ the simple three-dimensional module. Let $V = V_1 \oplus V_2 \oplus V_3$ and $W = W_1 \oplus W_2 \oplus W_3$. 
    
    By the Orbit-Stabiliser Theorem, we have that $ W_i \otimes W_i^* $ contains one or three one-dimensional modules in its direct sum decomposition. As it is four-dimensional, the latter is not possible. It follows that the $V_i$'s act transitively on the $W_i$'s. Therefore $W_i \otimes W_j \cong V_k \oplus U$ where $k$ is such that $W_i^* \cong W_j \otimes V_k$. Hence $ W_i \otimes W \cong V \oplus 3 U $ and we conclude that $\mathcal{K}(W_i) = 1$ as 
    $$ W_i \otimes (2V \oplus W) \cong V \oplus 2W \oplus 3 U $$
    Let us now compute $\mathcal{K}(U)$. We have that $U \otimes V_i \cong U$ and cancelling $W \otimes W_i \cong V \oplus 3U$ from $H \otimes W_i \cong 2 H$ yields that $U \otimes W_i \cong W$. We can now deduce from $U \otimes H \cong 3 H$ that $U \otimes U \cong V \oplus 2U$ and it follows that 
    $$ U \otimes (V_i \oplus 2 W_j \oplus U) \cong H $$
    So we conclude the Hopf algebra is of Knutson type.

    \underline{Case $(1^2, 2, 3^2)$}: Let $V_1, V_2$ be the one-dimensional modules, $W$ the two-dimensional simple module and $U_1, U_2$ the three-dimensional simple modules. Let $V = V_1 \oplus V_2$ and $U = U_1 \oplus U_2$. 
    
    We have that $W \otimes V_i \cong W$ so $W \otimes W$ contains $V$ precisely ones and by dimensions argument, $W \otimes W \cong V \oplus W$. If we cancel these two isomorphisms from $W \otimes H \cong 2 H$ and divide by three, we obtain that $W \otimes U \cong 2U$. As the direct sum decomposition of $W \otimes U_i$ depends on the action of $V_1$ and $V_2$ on the three-dimensional modules, we treat the two cases separately.
    
    Type I: Suppose that $V_1$ acts trivially on the $U_i$'s and $V_2$ by swapping them. As $V_2$ acts trivially on $W$ and therefore on $W \otimes U_i$, we must have that $W \otimes U_i \cong U$. So 
    $$ W \otimes ( V_i \oplus W \oplus 3 U_j) \cong H $$ 
    and $\mathcal{K}(W) = 1$. If we cancel all the previous from $U_i \otimes H \cong 3 H$ and divide by three, we obtain that $U_i \otimes U \cong V \oplus 2 W \oplus 2 U$. Hence
    $$ U_i \otimes (V \oplus U) \cong H $$
    and we conclude that the Hopf algebra is of Knutson type.
    
    Type II: If $V_1$ and $V_2$ act trivially on the $U_i's$ then $W \otimes U_i \cong 2 U_i$ and it follows that $\mathcal{K}(W) = 2$ as $W \otimes M$ always contains an even number of copies of $U_i$'s for any virtual module $M$ of $H$. We also have that $U_i \otimes (V \otimes U) \cong H$ so $\mathcal{K}(U_i) = 1$. The Hopf algebra is not of Knutson type.

    \item \textbf{Dimension $\mathbf{28}$}: The possible module structures are $(1^4, 2^6)$ or $(1^4, 2^2, 4)$. In the first case, it follows from Theorem \ref{thm two dim} that the Hopf algebra is of Knutson type. Let us now show that no Hopf algebra of dimension $28$ has a simple module of dimension four. The only two non-abelian groups of order $28$ are the dihedral group $\text{D}_{28}$ and the dicyclic group $\text{Dic}_7$ and they both have six simple two-dimensional modules. There are two non-commutative and non-cocommutative Hopf algebras of dimension $28$ \cite{natalepq2ii} and they have a commutative Hopf subalgebra of dimension $14$ \cite{gelaki}. So all their simple modules are of dimension one or two. 
    
    \item \textbf{Dimension $\mathbf{30}$}: $H$ must be group algebras or a dual of a group algebra \cite{natale24-30-40-42-54-56}. The latter case is trivial and the former follows from Corollary \ref{cor distinct primes} as $30$ is a product of distinct primes.
\end{itemize}
\end{proof}

\section{Savistkii's Conjecture}

In this last section, we present counter-examples to Savitkii's refinement of Knutson's Conjecture. Let us start by introducing Savitskii's pre-order on the set of virtual complex characters of a finite group $G$ and define the notion of minimal characters. 

\begin{definition}[Minimal Character]
The following is a pre-order on the set of virtual complex characters of strictly positive degree of a finite group $G$.
We say that $\chi \leq \varphi$ if $\chi(\text{id}) \leq \varphi(\text{id})$ and $\varphi(g) = 0 \implies \chi(g) = 0 \ \forall g \in G $. A character $\chi$ is minimal if for any character $\varphi$ we have that $\varphi \leq \chi \implies \chi \leq \varphi$. We denote by $\chi \lneq \varphi$ the case where $\chi \leq \varphi$, but not $\varphi \leq \chi$.
\end{definition}

Savitskii made the following refinement of Knutson's Conjecture \cite{savitskii}.

\begin{conjecture}[Savitskii]
A character is regular invertible if and only if it is minimal.
\end{conjecture}

Note that Knutson's Conjecture is stated for irreducible characters, whereas this conjecture is for all virtual characters. Our counter-examples in both directions are given by irreducible characters. This shows that the conjecture can not be fixed by restating it just for irreducible characters.

Let us first consider the forward implication. We start by making the following observation. Suppose that $\chi$ is a regular invertible irreducible character with regular inverse $\lambda$. If $\chi'$ is a virtual character such that $\chi' \leq \chi$ we have that 
$$ \chi' \otimes \lambda = \frac{\chi'(\text{id})}{\chi(\text{id})} \ \rho_\text{reg} $$
As the left-hand side is a virtual character and $\chi'(\text{id}) \leq \chi(\text{id})$ we must have that $\chi'(\text{id}) = \chi(\text{id})$. 

It follows that if there is no virtual character $\chi'$ such that $\chi' \lneq \chi$ and $\chi'(\text{id}) = \chi(\text{id})$ then $\chi$ is minimal and we conclude the forward implication of the conjecture holds for $\chi$. However, a character may be strictly smaller than another despite having the same dimension. 

So let us consider a group where this occurs to obtain a counter-example to the forward implication of Savitskii's Conjecture.

\begin{example}
Consider the dicyclic group ${\rm Dic}_6$ of order $24$. Its character table is
\begin{table}[h]
$$
\begin{array}{c|rrrrrrrrr}
  \rm class&\rm1&\rm2&\rm3&\rm4A&\rm4B&\rm4C&\rm6&\rm12A&\rm12B\cr
  \rm size&1&1&2&2&6&6&2&2&2\cr
\hline
  \chi_{1}&1&1&1&1&1&1&1&1&1\cr
  \chi_{2}&1&1&1&1&-1&-1&1&1&1\cr
  \chi_{3}&1&1&1&-1&-1&1&1&-1&-1\cr
  \chi_{4}&1&1&1&-1&1&-1&1&-1&-1\cr
  \chi_{5}&2&2&-1&2&0&0&-1&-1&-1\cr
  \chi_{6}&2&2&-1&-2&0&0&-1&1&1\cr
  \chi_{7}&2&-2&2&0&0&0&-2&0&0\cr
  \chi_{8}&2&-2&-1&0&0&0&1&\sqrt{3}&-\sqrt{3}\cr
  \chi_{9}&2&-2&-1&0&0&0&1&-\sqrt{3}&\sqrt{3}\cr
\end{array}
$$
\caption{Character table of ${\rm Dic}_6$ \cite{WinNT}}
\end{table}

Here $\chi_5$ is not minimal because $\chi_7 \lneq \chi_5$. Let $\lambda = \chi_1 + \chi_4 + \chi_5 + \chi_6 + \chi_7 + \chi_8 + \chi_9$. Then $\chi_5 \otimes \lambda = \rho_\text{reg}$. So we have found a non-minimal character that is regular invertible and disproved the forward implication. Further counter-examples can be found in the following three families of groups; ${\rm Dic}_n$ for even $n$ that is not a power of $2$, ${\rm D}_{2n}$ for $n$ divisible by $4$ and not a power $2$ and ${\rm C}_m \rtimes {\rm D}_8$ for odd $m \geq 3$.
\end{example}

So Savitskii's Conjecture does not hold in general and for irreducible characters in particular. We now disprove the reverse implication by presenting a group with an irreducible character that is minimal but not regular invertible.

It is proven in \cite{savitskii} that if the backwards implication of Savitskii's Conjecture holds for all Sylow $p$-subgroups of a group, then the conjecture holds for that group. It follows that if there is a counter-example to this implication in a group $G$, then there also is a counter-example in one of its Sylow $p$-subgroups and in particular, that if this implication holds for all $p$-groups, then it holds in general.

The idea is to take a minimal character $\chi$ and then compute the Savitskii Restriction of $\chi$ on all of its Sylow $p$-subgroups to obtain a character for each Sylow $p$-subgroup. Savitskii showed that these characters will be minimal and because we assumed that the conjecture is true on all its Sylow $p$-subgroups, we can find a regular inverse for each of them. We can therefore construct a regular inverse for $\chi$ using the Savitskii Induced Inverse algorithm. 

However, this proof is not valid as Savitskii Restriction does not preserve minimality. The proof assumes that $\chi_1 \lneq \chi_2 \implies \chi_1(\text{id}) < \chi_2(\text{id})$. However, we know that this is not always true as we saw with ${\rm Dic}_6$. The following example shows that Savitskii Restriction does not preserve minimality. 

\begin{example}
Consider $G = S_4$. This has Sylow $p$-subgroups $S_1 = {\rm D}_8$ and $S_2 = {\rm C}_3$. Let us consider the Savitskii Restriction to ${\rm D}_8$. 

\begin{table}[h]
\parbox{.45\linewidth}{
\centering
$$
\begin{array}{c|rrrrr}
  \rm class&\rm1&\rm2A&\rm2B&\rm3&\rm4\cr
  \rm size&1&3&6&8&6\cr
\hline
  \rho_{1}&1&1&1&1&1\cr
  \rho_{2}&1&1&-1&1&-1\cr
  \rho_{3}&2&2&0&-1&0\cr
  \rho_{4}&3&-1&-1&0&1\cr
  \rho_{5}&3&-1&1&0&-1\cr
\end{array}
$$
\caption{Character table of $S_4$}
}
\hfill
\parbox{.45\linewidth}{
\centering
$$
\begin{array}{c|rrrrr}
  \rm class&\rm1&\rm2A&\rm2B&\rm2C&\rm4\cr
  \rm size&1&1&2&2&2\cr
\hline
  \lambda_{1}&1&1&1&1&1\cr
  \lambda_{2}&1&1&-1&1&-1\cr
  \lambda_{3}&1&1&1&-1&-1\cr
  \lambda_{4}&1&1&-1&-1&1\cr
  \lambda_{5}&2&-2&0&0&0\cr
\end{array}
$$
\caption{Character table of ${\rm D}_8$}
}
\end{table}

We start by claiming that $\rho_3$ on $S_4$ is a minimal character. There is no two-dimensional character with more zeros than $\rho_3$ and any one-dimensional character has fewer zeros. We have that $\text{gcd}(\rho_2(\text{id}), p_1^3) = 2$ so we need to find $u_1, v_1 \in \mathbb{Z}$ such that $2 u_1 + 8 v_1 = 2$. The solutions to this are $(u_1, v_1) = (1-4a, a)$. With our convention of picking $u$ minimal positive, we would choose $(u_1, v_1) = (1, 0)$, but let us show that the resulting character is not minimal independent of this choice. This gives us that the Savitskii Restriction of $\rho_3$ to ${\rm D}_8$ is 
$$ \nu_1 = (1-4a) \rho_3|_{D_8} + a \rho_{\text{reg}} = (1 - 4a) (\lambda_1 + \lambda_3) + a \lambda_\text{reg} $$
So for any $a \in \mathbb{Z}$ its dimension is at least $2$ and is zero precisely on the conjugacy classes $2C$ and $4$. Therefore, $\lambda_5 \lneq \nu_1$ and $\nu_1$ and we conclude that $\nu_1$ is not minimal.

\end{example}

So Savitskii's proof is not valid and it is, in principle, possible that the backwards implication is true for all $p$-groups but not in general. Our goal is now to find a counter-example for this implication. If Knutson's Conjecture is true for a given group, then this implication will also be trivially true for that group. So we start by considering all the counter-examples that we found for Knutson's Conjecture in the last section.

For ${\rm SL}_2(\mathbb{F}_5)$, we see that the two four-dimensional characters are not minimal because any of the two-dimensional characters is strictly smaller. So the backwards implication is true for irreducible characters of ${\rm SL}_2(\mathbb{F}_5)$ as expected. For ${\rm PSL}_2(\mathbb{F}_{11})$, the irreducible characters that are not regular invertible are precisely the $10$-dimensional characters and they are the only two non-minimal because any of the $5$-dimensional characters is strictly smaller. In fact, Savitskii claimed that his conjecture holds for ${\rm SL}_2(\mathbb{F}_q)$ and ${\rm PSL}_2(\mathbb{F}_q)$.

We should therefore consider one of our original counter-examples to Knutson's Conjecture.

\begin{example}
The character table of ${\rm D}_{12} \rtimes S_3$ is
\begin{table}[h]
$$
\begin{array}{c|rrrrrrrrrrrrrrr}
  \rm class&\rm1&\rm2A&\rm2B&\rm2C&\rm3A&\rm3B&\rm3C&\rm4&\rm6A&\rm6B&\rm6C&\rm6D&\rm6E&\rm6F&\rm6G\cr
  \rm size&1&1&6&6&2&2&4&18&2&2&4&6&6&6&6\cr
\hline
  \rho_{1}&1&1&1&1&1&1&1&1&1&1&1&1&1&1&1\cr
  \rho_{2}&1&1&1&-1&1&1&1&-1&1&1&1&-1&1&1&-1\cr
  \rho_{3}&1&1&-1&-1&1&1&1&1&1&1&1&-1&-1&-1&-1\cr
  \rho_{4}&1&1&-1&1&1&1&1&-1&1&1&1&1&-1&-1&1\cr
  \rho_{5}&2&2&-2&0&2&-1&-1&0&2&-1&-1&0&1&1&0\cr
  \rho_{6}&2&-2&0&0&2&2&2&0&-2&-2&-2&0&0&0&0\cr
  \rho_{7}&2&2&2&0&2&-1&-1&0&2&-1&-1&0&-1&-1&0\cr
  \rho_{8}&2&2&0&2&-1&2&-1&0&-1&2&-1&-1&0&0&-1\cr
  \rho_{9}&2&2&0&-2&-1&2&-1&0&-1&2&-1&1&0&0&1\cr
  \rho_{10}&2&-2&0&0&-1&2&-1&0&1&-2&1&-\alpha&0&0&\alpha\cr
  \rho_{11}&2&-2&0&0&2&-1&-1&0&-2&1&1&0&\alpha&-\alpha&0\cr
  \rho_{12}&2&-2&0&0&2&-1&-1&0&-2&1&1&0&-\alpha&\alpha&0\cr
  \rho_{13}&2&-2&0&0&-1&2&-1&0&1&-2&1&\alpha&0&0&-\alpha\cr
  \rho_{14}&4&4&0&0&-2&-2&1&0&-2&-2&1&0&0&0&0\cr
  \rho_{15}&4&-4&0&0&-2&-2&1&0&2&2&-1&0&0&0&0\cr
\end{array}
$$
where $\alpha = \sqrt{3} i$. 
\caption{Character table of ${\rm D}_{12}{\rtimes}S_3$ \cite{WinNT}}
\end{table}

Code \ref{span} yields that the not regular invertible characters of this group are precisely $\rho_{14}$ and $\rho_{15}$, the $4$-dimensional ones. We can show that these characters are minimal by showing that no integer linear combination of the irreducible characters is strictly smaller than them. To do this, we have to amend Code \ref{minimal} appropriately to show that no character of dimension $1$, $2$ or $3$ is zero whenever $\rho_{14}$ and $\rho_{15}$ are zero. We can also show with the same code that every $4$-dimensional character that is zero whenever $\rho_{14}$ and $\rho_{15}$ are zero has no extra zeros. We, therefore, conclude that $\rho_{14}$ and $\rho_{15}$ are minimal. 
\end{example}

Note also that Savitskii's Conjecture is true on all Sylow $p$-subgroups of ${\rm D}_{12}{\rtimes}S_3$, as these are ${\rm D}_8$ and ${\rm C}_3 \times {\rm C}_3$. So we show that even if the backwards implication is true on all Sylow $p$-subgroups of a group $G$, this implication is not necessarily true on $G$.

\section{Appendix}
\begin{code} \label{span}
The following code checks which irreducible characters of a group $G$ are regular invertible.
\includegraphics[scale = 0.67]{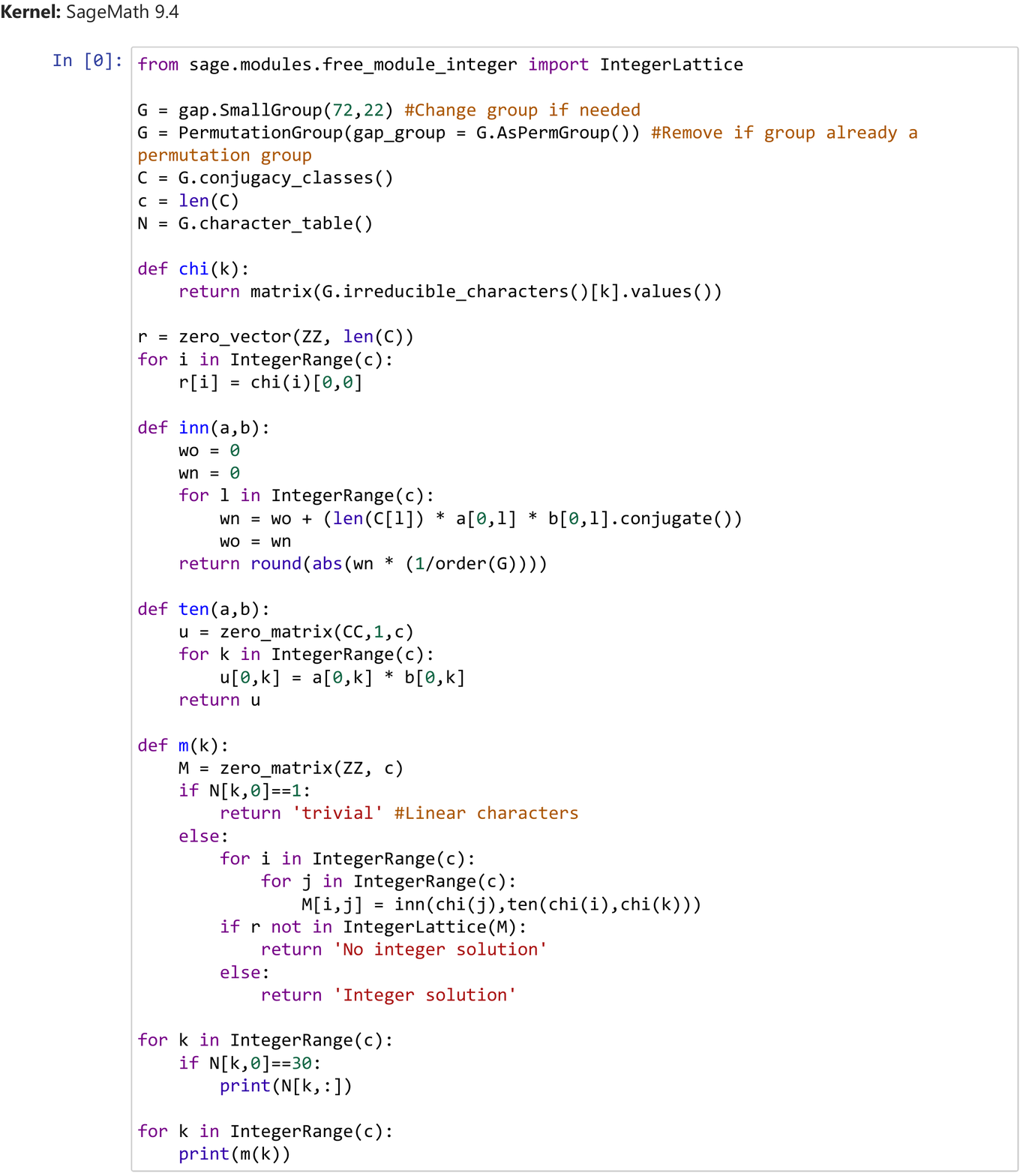}
\end{code}

\newpage
\begin{code} \label{code index}
The following code computes the Knutson Index of all irreducible characters of $G$.
\includegraphics[scale = 0.67]{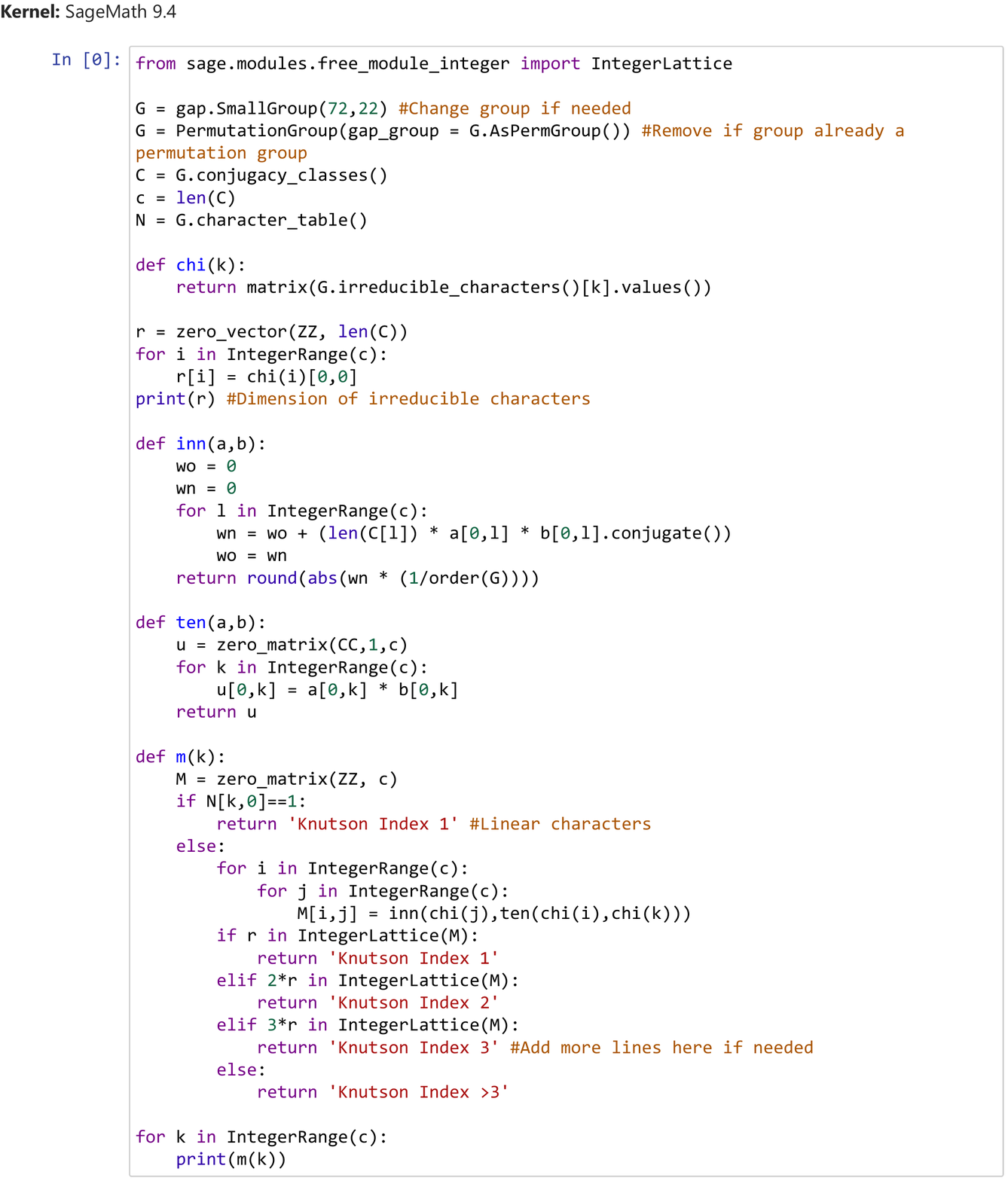}
\end{code}

\newpage
\begin{code} \label{minimal} 
With the following code, we can check whether an irreducible character of $G$ is minimal.
\includegraphics[scale = 0.67]{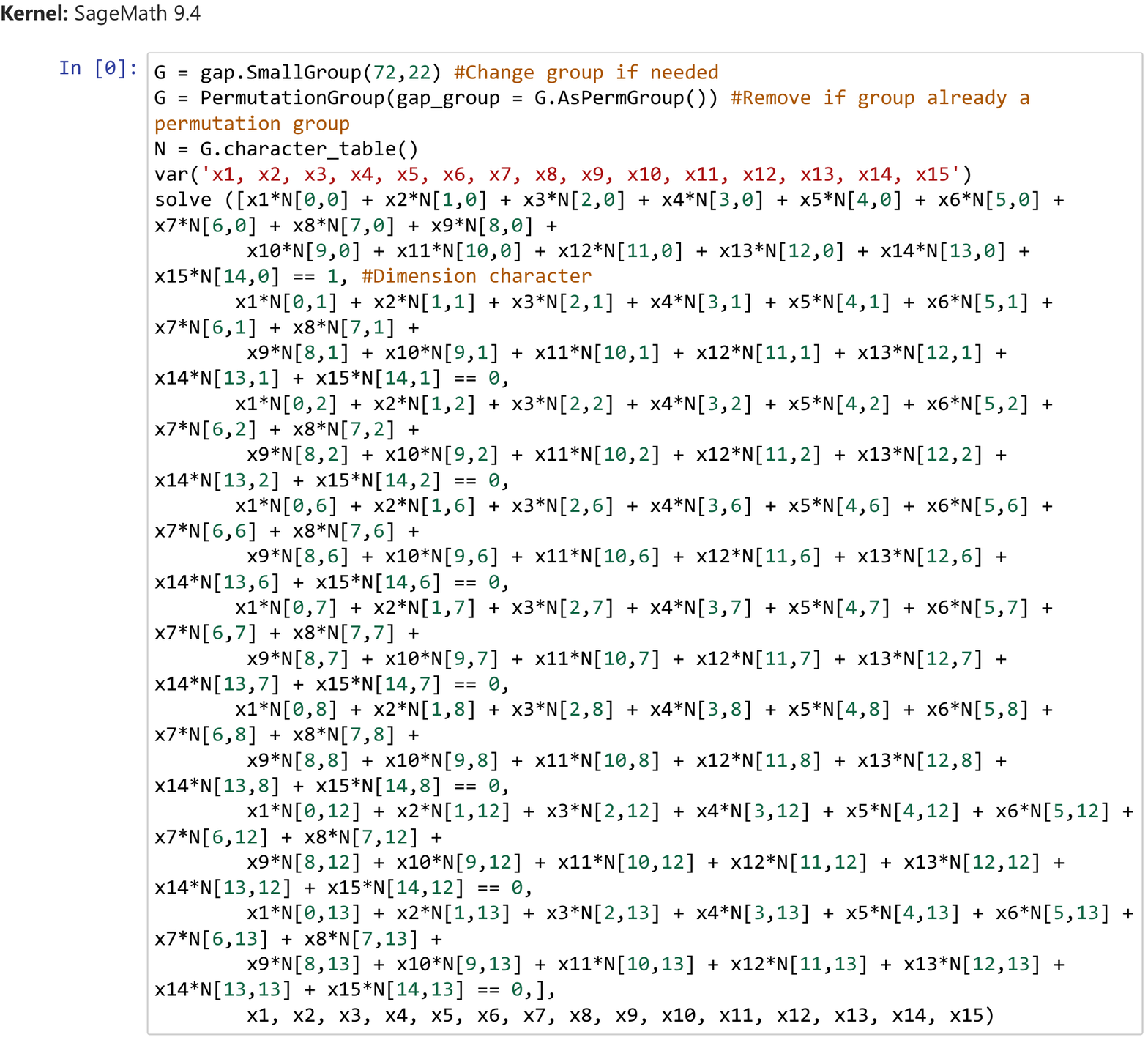}
\end{code}

\newpage
\bibstyle{numeric}
\printbibliography
\end{document}